\documentclass{amsart}

\usepackage{amsmath}
\usepackage{amsthm}
\usepackage{amssymb}
\usepackage[colorlinks=true,linkcolor=blue,citecolor=blue,urlcolor=blue]{hyperref}

\DeclareMathOperator{\Gal}{Gal}
\DeclareMathOperator{\Sym}{Sym}
\DeclareMathOperator{\Stab}{Stab}
\DeclareMathOperator{\Fix}{Fix}

\newcommand{\PP}{\mathbb{P}}
\newcommand{\NN}{\mathbb{N}}
\newcommand{\QQ}{\mathbb{Q}}
\newcommand{\ZZ}{\mathbb{Z}}
\newcommand{\CC}{\mathbb{C}}
\newcommand{\OO}{\mathcal{O}}

\theoremstyle{plain}
\newtheorem{thm}{\protect\theoremname}
\theoremstyle{definition}
\newtheorem{definition}[thm]{\protect\definitionname}
\theoremstyle{plain}
\newtheorem{lemma}[thm]{\protect\lemmaname}
\theoremstyle{remark}
\newtheorem{rem}[thm]{\protect\remarkname}
\theoremstyle{plain}
\newtheorem{prop}[thm]{\protect\propositionname}

\providecommand{\definitionname}{Definition}
\providecommand{\lemmaname}{Lemma}
\providecommand{\propositionname}{Proposition}
\providecommand{\remarkname}{Remark}
\providecommand{\theoremname}{Theorem}

\numberwithin{thm}{section}

\title{Counting Polynomials via Galois Actions on Root Subsets}
\author{Or Ben-Porath}
\subjclass[2020]{Primary 11R32, 11R09; Secondary 11R45, 20B05}
\date{\today}

\begin{document}
\begin{abstract}
    This paper studies the number of monic integer polynomials $f$ of height at most $H$ whose Galois group, endowed with the action on the roots, is isomorphic to a prescribed permutation group $(G,\Omega)$.
    New upper bounds are obtained for several families of groups:
    transitive subgroups of the wreath product $S_m\wr S_r$ in the primitive action;
    $k$-homogeneous subgroups of $S_m$ in the action on $k$-subsets of $\{1,\ldots,m\}$;
    $k$-transitive subgroups of $S_m$ in the action on $k$-tuples of distinct elements of $\{1,\ldots,m\}$.
    Almost all finite groups in their regular permutation representation are also treated.
\end{abstract}
\maketitle

\section{Introduction}
\label{sec:introduction}
Probabilistic Galois theory studies the distribution of Galois groups of random algebraic objects.
A central instance is the large box model, where one considers monic integer polynomials of fixed degree and growing height.

A finite group $G$ acting faithfully on a finite set $\Omega$ is called a permutation group and denoted by $(G,\Omega)$.
Denote the height of a polynomial $f = \sum_i a_i x^i \in \ZZ[x]$ by $H(f) = \max_i |a_i|$. 
For $n\in \NN$ and $H\ge 1$, set
\[
\begin{aligned}
    B_n(H) &:= \{f\in \ZZ[x]: f \textnormal{ monic}, \deg f=n, H(f) \le H\},\\
    b_n(H) &:= |B_n(H)|=(2\lfloor H \rfloor+1)^n\sim (2H)^n.
\end{aligned}
\]
To each separable $f\in B_n(H)$ attach the permutation group $(G_f,\Omega_f)$, where
\[
    \Omega_f:=\{\alpha\in \CC:f(\alpha)=0\},\quad G_f:=\Gal\bigl(\QQ(\Omega_f)/\QQ\bigr),
\]
and $G_f$ acts on $\Omega_f$ naturally.
Given a finite permutation group $(G,\Omega)$, denote
\[
    B(H;G,\Omega):=\{f\in B_{|\Omega|}(H):(G_f,\Omega_f)\cong (G,\Omega)\}.
\]

Write $[n]:=\{1,\ldots,n\}$. Hilbert's irreducibility theorem implies that
\[
    \left|B(H;S_n,[n])\right|\sim b_n(H).
\]
For $n\ge 3$, Chela \cite{chela1963reducible} proved that
\[
    \left|\{f\in B_n(H): f \text{ reducible}\}\right|\sim \left|B(H;S_{n-1},[n])\right|\sim c_n H^{n-1},
\]
for an explicit constant $c_n$.
Van der Waerden \cite{van1936seltenheit} conjectured that
\[
    \left|\{f\in B_n(H):(G_f,\Omega_f)\not\cong (S_n,[n])\}\right|\sim \left|B(H;S_{n-1},[n])\right|.
\]
Progress towards this conjecture was made in a series of subsequent works \cite{knobloch1955hilbertschen, gallagher1973large, zywina2010hilbert, dietmann2013probabilistic, chow2023towards, anderson2023quantitative}.
Finally, Bhargava \cite[Theorem~1]{bhargava2025galois} proved the conjecture up to a constant, by showing that
\[
    \left|B(H;G,[n])\right|\ll H^{n-1}
\]
for every proper subgroup $G<S_n$.

In this paper we obtain new upper bounds on $\left|B(H;G,\Omega)\right|$ for several families of permutation groups $(G,\Omega)$.

\subsection*{Primitive wreath product actions}

Let $m,r,k\in \NN$ with $m\ge 2$, $k\le m/2$, and $(r,k)\neq (1,1)$.
Set
\[
    \binom{[m]}{k}:=\{A\subset [m]:|A|=k\},\qquad \Omega^P(m,r,k):=\binom{[m]}{k}^r,
\]
so that elements of $\binom{[m]}{k}$ are $k$-subsets of $[m]$, and elements of $\Omega^P(m,r,k)$ are $r$-tuples of $k$-subsets of $[m]$.
Put $n:=|\Omega^P(m,r,k)|=\binom{m}{k}^r$.
The wreath product $S_m\wr S_r$ (see Definition~\ref{def:wreath-product}) acts naturally on $\Omega^P(m,r,k)$ in its primitive action (see Definition~\ref{def:actions-of-wreath-product}).
Following \cite{bhargava2025galois} we call a transitive subgroup of $(S_m\wr S_r,\Omega^P(m,r,k))$ a primitive non-elemental permutation group of type $(r,m,k)$.
Bhargava \cite[Corollary 3]{bhargava2025galois} proved that for every primitive non-elemental permutation group $G$ of type $(r,m,k)$,
\begin{equation}
\label{eq:bhargava-primitive-wreath-product}
    \left|B(H;G,\Omega^P(m,r,k))\right|\ll H^{c\sqrt{n}(\log n)^2}=H^{c\binom{m}{k}^{r/2}r^2 \log^2 \binom{m}{k}},
\end{equation}
for some absolute constant $c$. 
Our first result improves this bound for $r>1$.

\begin{thm}
\label{thm:primitive-wreath-product}
    Fix $\epsilon>0$, integers $m,r\ge 2$, and $1\le k\le m/2$.
    Let $G\le S_m\wr S_r$ be a primitive non-elemental permutation group of type $(r,m,k)$.
    Then
    \begin{equation}
    \label{eq:thm-primitive-wreath-product-bound}
        \left|B(H;G,\Omega^P(m,r,k))\right|\ll H^{r\left(\binom{m}{k}+r-1\right)+1+\epsilon}.
    \end{equation}
\end{thm}

We compare the exponents $e_1$ and $e_2$ of $H$ in \eqref{eq:bhargava-primitive-wreath-product} and \eqref{eq:thm-primitive-wreath-product-bound}, respectively.
Since $\binom{m}{k}=n^{1/r}$ and $r=\log_{\binom{m}{k}} n \le\log_2 n$,
\[
    e_2=r\left(\binom{m}{k}+r-1\right)+1+\epsilon \ll n^{1/r}\log n+(\log n)^2.
\]
In particular,
\[
    \frac{e_1}{e_2} \gg \frac{\sqrt{n}(\log n)^2}{n^{1/r}\log n+(\log n)^2} \gg
    \begin{cases}
        \log n, & r=2, \\
        n^{1/2-1/r}, & r>2.
    \end{cases}
\]
Hence, for $n$ sufficiently large, \eqref{eq:thm-primitive-wreath-product-bound} is sharper than \eqref{eq:bhargava-primitive-wreath-product}.

\subsection*{$k$-homogeneous and $k$-transitive actions}

Primitive non-elemental permutation groups of type $(1,m,k)$ are subgroups $G\le S_m$ acting transitively on $k$-subsets of $[m]$, which are classically called $k$-homogeneous groups.
Specializing \eqref{eq:bhargava-primitive-wreath-product} to $r=1$ gives
\begin{equation}
\label{eq:bhargava-k-homogeneous}
    \left|B\left(H;G,\binom{[m]}{k}\right)\right|\ll H^{c\binom{m}{k}^{1/2}\log^2\binom{m}{k}}.
\end{equation}
Since $(S_m,\binom{[m]}{k})\cong (S_m,\binom{[m]}{m-k})$ via the complement map, there is no loss of generality in assuming $k\le m/2$.
Our next result improves this bound.

\begin{thm}
\label{thm:k-homogeneous}
    Fix $\epsilon>0$ and $1\le k \le m/2$.
    Let $G\le S_m$ be $k$-homogeneous.
    Then
    \begin{equation}
    \label{eq:thm-k-homogeneous-bound}
        \left|B\left(H;G,\binom{[m]}{k}\right)\right|\ll H^{mk+1+\epsilon}.
    \end{equation}
\end{thm}

For $k=1$, the bound is immediate since $\left|B(H;G,[m])\right|\le b_m(H)\asymp H^m$.
We compare the exponents $e_3$ and $e_4$ of $H$ in \eqref{eq:bhargava-k-homogeneous} and \eqref{eq:thm-k-homogeneous-bound}, respectively.
Using the well-known bound $\binom{m}{k}\ge (m/k)^k$,
\[
    \frac{e_3}{e_4} \gg \frac{(m/k)^{k/2}\cdot k^2\log^2(m/k)}{mk} = (m/k)^{k/2-1}\log^2(m/k),
\]
Hence, for $m$ sufficiently large and any $2\le k\le m/2$, \eqref{eq:thm-k-homogeneous-bound} is sharper than \eqref{eq:bhargava-k-homogeneous}.

We next consider $k$-transitive actions, a strengthening of $k$-homogeneity.
For $m\ge 2$ and $1\le k\le m$, let
\[
    [m]_k:=\{(v_1,\ldots,v_k)\in [m]^k: v_i\neq v_j \text{ for } i\neq j\}
\]
denote the set of $k$-tuples of distinct elements of $[m]$.
Its cardinality is the falling factorial $(m)_k=\frac{m!}{(m-k)!}$.
The symmetric group $S_m$ acts on $[m]_k$ coordinatewise.
A subgroup $G\le S_m$ is called $k$-transitive if it acts transitively on $[m]_k$.
Clearly, a $k$-transitive subgroup of $S_m$ is also $k$-homogeneous.
In the other direction, by \cite[Theorem 2]{livingstone1965transitivity}, for $k\le m/2$, every $k$-homogeneous subgroup of $S_m$ is $(k-1)$-transitive.
\begin{thm}
\label{thm:k-transitive}
    Fix $\epsilon>0$ and $1\le k\le m$.
    Let $G\le S_m$ be $k$-transitive.
    Then
    \[
        \left|B(H;G,[m]_k)\right|\ll H^{m+1+\epsilon}.
    \]
\end{thm}

For $k=1$, Theorem~\ref{thm:k-transitive} is immediate, as $\left|B(H;G,[m])\right|\le b_m(H)\asymp H^m$.
The case $m=k=2$ is similarly immediate, as $\left|B(H;G,[2]_2)\right|\le b_2(H)\asymp H^2$.
We are unaware of previous bounds specific to $\left|B(H;G,[m]_k)\right|$ for $k$-transitive groups.
For any $k\ge 2$, since $G$ is a proper transitive subgroup of $S_{(m)_k}$, Bhargava \cite[Theorem~1]{bhargava2025galois} gives
\begin{equation}
\label{eq:bhargava-k-transitive}
    \left|B(H;G,[m]_k)\right|\ll H^{(m)_k-1}.
\end{equation}
For any $m\ge 3$ and $k\ge 2$, Theorem~\ref{thm:k-transitive} gives a sharper bound than \eqref{eq:bhargava-k-transitive}.

We note that the classification of finite simple groups implies that, other than $A_m$ and $S_m$, there are no $k$-transitive (hence $k$-homogeneous) subgroups of $S_m$ for $k$ larger than some absolute constant.
Nevertheless, Theorems~\ref{thm:k-homogeneous} and~\ref{thm:k-transitive} are not vacuous in this regime, as they give new bounds for $|B(H;G,\binom{[m]}{k})|$ and $|B(H;G,[m]_k)|$ for $G=A_m, S_m$.

\subsection*{Regular actions}
Recall that a permutation group $(G,\Omega)$ is called regular if the action is transitive and $|G|=|\Omega|$.
In this case $(G,\Omega)\cong (G,G)$, where $G$ acts on itself by left multiplication.
Bhargava \cite[Corollary~6]{bhargava2025galois} proved that for every regular permutation group $(G,\Omega)$,
\begin{equation}
\label{eq:bhargava-regular}
    \left|B(H;G,\Omega)\right|\ll H^{\frac{3|G|}{11}+1.164}.
\end{equation}

Since $S_m$ is $m$-transitive and $|[m]_m|=(m)_m=m!$, the permutation group $(S_m,[m]_m)$ is regular.
Theorem~\ref{thm:k-transitive} then gives
\[
    \left|B(H;S_m,[m]_m)\right|\ll H^{m+1+\epsilon},
\]
whereas \eqref{eq:bhargava-regular} only gives
\[
    \left|B(H;S_m,[m]_m)\right|\ll H^{\frac{3m!}{11}+1.164}.
\]

For a finite group $G$, let $\mu(G)$ denote the \emph{minimal faithful permutation degree} of $G$, that is,
\[
    \mu(G):=\min\{m\in \NN: G \textnormal{ admits a faithful action on an }m\textnormal{-set}\}.
\]
The following theorem generalizes the bound for $(S_m,[m]_m)$ to all regular actions.
\begin{thm}
\label{thm:regular}
    Fix $\epsilon>0$.
    Let $(G,\Omega)$ be a regular permutation group with minimal faithful permutation degree $\mu=\mu(G)$.
    Then
    \begin{equation}
    \label{eq:thm-regular-bound}
        \left|B(H;G,\Omega)\right|\ll H^{\mu+1+\epsilon}.
    \end{equation}
\end{thm}

We compare the exponents of $H$ in \eqref{eq:bhargava-regular} and \eqref{eq:thm-regular-bound}.
Let $\iota(G)$ denote the minimal index of a cyclic subgroup of prime-power order in $G$.
Babai, Goodman and Pyber \cite[Theorem~1.1]{babai1993faithful} proved that there exists a function $\psi$ with $\psi(n)\to\infty$ as $n\to\infty$ such that
\[
    \mu(G)\le \frac{|G|}{\psi\bigl(\iota(G)\bigr)},
\]
for every finite group $G$.
In particular, there exists an absolute constant $C$ such that $\mu(G)<\frac{3|G|}{11}$, unless $G$ contains a cyclic subgroup of prime-power order and index at most $C$.
Hence, for all other groups, \eqref{eq:thm-regular-bound} is sharper than \eqref{eq:bhargava-regular}.

\subsection*{Methods}
The approach resembles the strategy of \cite{lemke2020upper}: 
first bound the number of fields arising from polynomials in $B(H;G,\Omega)$, then bound the number of such polynomials with a given splitting field.
For $f\in B(H;G,\Omega)$, let $K_f$ denote its splitting field.
For each fixed $G$-extension $K/\QQ$, the number of $f\in B(H;G,\Omega)$ with $K_f=K$ is $\ll H^{1+\epsilon}$ (cf. \cite[Proposition~2.2]{lemke2020upper}).
Thus it suffices to bound the number of possible splitting fields $K_f$ arising from $B(H;G,\Omega)$.

To this end, for each action $(G,\Omega)$ considered above, we construct a family $\Lambda$ of subsets of $\Omega$ that is closed under the induced action of $G$ on the power set $\mathcal{P}(\Omega)$, and on which $G$ acts faithfully.
Let $\Lambda^{(1)},\ldots,\Lambda^{(s)}$ be the orbits of $(G,\Lambda)$, and for each $\ell$ fix a representative $A_\ell\in\Lambda^{(\ell)}$.
For $f\in B(H;G,\Omega)$, we show that $K_f$ is the compositum of the fixed fields $E_\ell:=K_f^{\Stab_G(A_\ell)}$, and use the coefficients of $\prod_{\alpha\in A_\ell}(x-\alpha)$ to produce a generator of $E_\ell$ whose height is bounded by a power of $H$.
This reduces the counting problem for $(G,\Omega)$ to the actions $(G,\Lambda^{(\ell)})$.

\subsection*{Organization}
In Section~\ref{sec:group-actions} we prove the group-theoretic lemmas needed to construct the families $\Lambda$.
In Section~\ref{sec:counting} we show how such a family $\Lambda$ reduces the polynomial count to simpler actions.
In Section~\ref{sec:proofs} we combine these tools to prove Theorems~\ref{thm:primitive-wreath-product}-\ref{thm:regular}.

\subsection*{Acknowledgements}
The author thanks Danny Neftin for suggesting the primitive wreath product problem and for many insightful discussions.
The author thanks Daniele Garzoni for many helpful references and insights on the minimal faithful permutation degree.
The author thanks Lior Bary-Soroker for his guidance and many helpful comments.

The author is supported by the Israel Science Foundation (grant no.~366/23).

\section{Group actions}
\label{sec:group-actions}

\subsection*{Definitions and notation}

A \emph{finite group action} $(G,\Omega)$ consists of a finite group $G$, a finite set $\Omega$, and an action homomorphism $\pi:G\to\Sym(\Omega)$.
All group actions in this paper are tacitly assumed to be finite.
For $\sigma\in G$ and $\omega\in\Omega$, abbreviate $\pi(\sigma)(\omega)$ to $\sigma(\omega)$.
An action $(G,\Omega)$ is called a \emph{permutation group} if it is faithful, that is, if $\pi$ is injective.
Every group action $(G,\Omega)$ gives rise to a permutation group $(\pi(G),\Omega)$, called the \emph{permutation representation} of $(G,\Omega)$.

A \emph{morphism} $(\varphi,\varphi^*):(G_1,\Omega_1)\to (G_2,\Omega_2)$ consists of a group homomorphism $\varphi:G_1\to G_2$ and a map $\varphi^*:\Omega_1\to \Omega_2$ satisfying
\[
    \varphi(\sigma)\bigl(\varphi^*(\omega)\bigr)=\varphi^*\bigl(\sigma(\omega)\bigr)
\]
for all $\sigma\in G_1$ and $\omega\in \Omega_1$.
An \emph{isomorphism} is a morphism $(\varphi,\varphi^*)$ with $\varphi$ a group isomorphism and $\varphi^*$ a bijection.
For $H\le G$, the inclusion $H\hookrightarrow G$ defines a morphism $(H,\Omega)\to (G,\Omega)$.

For $\sigma\in G$ and $A\subset \Omega$, let $\sigma(A):=\{\sigma(\omega):\omega\in A\}$.
This defines an action $(G,\mathcal{P}(\Omega))$, where $\mathcal{P}(\Omega):=\{A\subset\Omega\}$ denotes the power set of $\Omega$.
The maps $\varphi(\sigma)=\sigma$ and $\varphi^*(\omega)=\{\omega\}$ define a morphism $(G,\Omega)\to(G,\mathcal{P}(\Omega))$.
$A\subset\Omega$ is said to be \emph{$(G,\Omega)$-stable} if $\sigma(A)=A$ for all $\sigma\in G$.
In that case, the \emph{restricted action} $(G,A)$ is well defined, and the inclusion $A\hookrightarrow\Omega$ defines a morphism $(G,A)\to(G,\Omega)$.
In particular, a family $\Lambda\subset\mathcal{P}(\Omega)$ is $(G,\mathcal{P}(\Omega))$-stable if and only if $\sigma(A)\in\Lambda$ for all $\sigma\in G$ and $A\in\Lambda$.

\begin{lemma}
\label{lem:equivariant-family-isomorphism}
    Let $(G,\Omega)$ and $(G,\Delta)$ be group actions, and let 
    \[
        \Lambda=\{\Lambda_\delta:\delta\in\Delta\}\subset\mathcal{P}(\Omega).
    \]
    If the map $\iota\colon \Delta\to \Lambda$, $\delta\mapsto \Lambda_\delta$ is injective and $\sigma(\Lambda_\delta)\subset \Lambda_{\sigma(\delta)}$ for all $\sigma\in G$ and $\delta\in \Delta$, then $\Lambda$ is $(G,\mathcal{P}(\Omega))$-stable and $(\mathrm{id},\iota):(G,\Delta)\to(G,\Lambda)$ is an isomorphism.
\end{lemma}

\begin{proof}
    Applying the hypothesis to $\sigma^{-1}$ and $\sigma(\delta)$ gives $\sigma^{-1}(\Lambda_{\sigma(\delta)})\subset \Lambda_\delta$, hence $\Lambda_{\sigma(\delta)}\subset \sigma(\Lambda_\delta)$.
    Combined with $\sigma(\Lambda_\delta)\subset \Lambda_{\sigma(\delta)}$, this gives $\sigma(\Lambda_\delta)=\Lambda_{\sigma(\delta)}$.
    In particular, $\Lambda$ is $(G,\mathcal{P}(\Omega))$-stable.
    Since $\iota$ is injective, it is a bijection $\Delta\to\Lambda$, and the equality $\sigma(\Lambda_\delta)=\Lambda_{\sigma(\delta)}$ shows that $(\mathrm{id},\iota)$ is a morphism.
\end{proof}

Denote the $G$-saturation of $A$, the setwise stabilizer of $A$, and the pointwise stabilizer of $A$ by
\begin{align*}
    G\cdot A &:= \{\sigma(\omega):\sigma\in G,\omega\in A\}, \\
    \Stab_G(A) &:= \{\sigma\in G:\sigma(A)=A\}, \\
    \Fix_G(A) &:= \{\sigma\in G: \forall \omega\in A,\ \sigma(\omega)=\omega\}.
\end{align*}
For $\omega\in \Omega$, by abuse of notation denote the orbit of $\omega$ and the stabilizer of $\omega$ by
\[
    G\cdot \omega:=G\cdot \{\omega\},\quad \Stab_G(\omega):=\Stab_G(\{\omega\}).
\]

For $m\in \NN$, denote $[m]=\{1,\ldots,m\}$.
For $1\le k\le m$, set
\[
    \binom{[m]}{k}:=\{A\subset [m]:|A|=k\},\qquad
    [m]_k:=\{(v_1,\ldots,v_k)\in [m]^k: v_i\neq v_j\ \text{for}\ i\neq j\},
\]
whose cardinalities are $\binom{m}{k}=\frac{m!}{k!(m-k)!}$ and $(m)_k=\frac{m!}{(m-k)!}$, respectively.
The symmetric group $S_m$ acts on $\binom{[m]}{k}$ via the action on $\mathcal{P}([m])$ and on $[m]_k$ coordinatewise.

\subsection*{Wreath products}

\begin{definition}
\label{def:wreath-product}
Let $m,r\geq 2$.
Let $S_r$ act on $S_m^r$ by permuting the factors:
for $\rho\in S_r$ and $\tau=(\tau_1,\ldots,\tau_r)\in S_m^r$, set
\[
    \rho\cdot\tau:=(\tau_{\rho^{-1}(1)},\ldots,\tau_{\rho^{-1}(r)}).
\]
Define the \emph{wreath product} of $S_m$ and $S_r$ as the semidirect product of $S_m^r$ and $S_r$ with respect to this action
\[
    S_m\wr S_r:=S_m^r\rtimes S_r.
\]
Thus, every element of $S_m\wr S_r$ may be uniquely written as $(\tau,\rho)$ with $\tau\in S_m^r$ and $\rho\in S_r$, and the group operation is
\[
    (\tau_1,\rho_1)(\tau_2,\rho_2)=(\tau_1(\rho_1\cdot\tau_2), \rho_1\rho_2).
\]
Set $R:=S_m\wr S_r$.
\end{definition}

\begin{definition}
\label{def:actions-of-wreath-product}
    The group $R$ carries several natural actions.
    \begin{enumerate}
        \item \emph{Top action}:
        Set $\Omega^T(m,r):=[r]$.
        For $i\in \Omega^T(m,r)$ and $(\tau,\rho)\in R$, define
        \[
            (\tau,\rho)(i):=\rho(i).
        \]
        \item \emph{Primitive action}:
        For $1\le k\le m/2$, set $\Omega^P(m,r,k):=\binom{[m]}{k}^r$.
        For $(A_1,\ldots,A_r)\in \Omega^P(m,r,k)$ and $(\tau,\rho)\in R$ with $\tau=(\tau_1,\ldots,\tau_r)$, define
        \[
            (\tau,\rho)(A_1,\ldots,A_r):=(\tau_1(A_{\rho^{-1}(1)}),\ldots,\tau_r(A_{\rho^{-1}(r)})).
        \]
        \item \emph{Imprimitive action}:
        Set $\Omega^I(m,r):=[r]\times[m]$.
        For $(i,j)\in \Omega^I(m,r)$ and $(\tau,\rho)\in R$ with $\tau=(\tau_1,\ldots,\tau_r)$, define
        \[
            (\tau,\rho)(i,j):=(\rho(i),\tau_{\rho(i)}(j)).
        \]
    \end{enumerate}
    For $\omega\in \Omega^T(m,r)\cup\Omega^P(m,r,k)\cup\Omega^I(m,r)$, the action of $R$ on $\omega$ is understood from context.

    When $m,r$ are clear, write $\Omega^T,\Omega^I,\Omega^P$ for $\Omega^T(m,r),\Omega^I(m,r),\Omega^P(m,r,1)=[m]^r$, respectively.
\end{definition}

\begin{lemma}
\label{lem:imprimitive-action-faithful}
    $(R,\Omega^I)$ is faithful.
\end{lemma}

\begin{proof}
    Let $(\tau,\rho)\in R$ fix every $(i,j)\in \Omega^I$.
    Then $\rho(i)=i$ for all $i\in[r]$.
    Hence $\rho=\mathrm{id}$.
    The equality $(\tau,\rho)(i,j)=(i,j)$ then becomes $\tau_i(j)=j$ for all $i\in [r]$ and $j\in [m]$.
    Hence $\tau_i=\mathrm{id}$ for all $i$, so $(\tau,\rho)=1$.
\end{proof}

\begin{lemma}
\label{lem:primitive-transitive-imprimitive-orbit}
    Let $G\le R$ with $(G,\Omega^P)$ transitive.
    Then for every $(i,j)\in\Omega^I$,
    \[
        G\cdot (i,j)=(G\cdot i)\times[m].
    \]
\end{lemma}

\begin{proof}
    Clearly, $G\cdot (i,j)\subset (G\cdot i)\times[m]$.

    It suffices to find $i_0\in G\cdot i$ such that $\{i_0\}\times[m]\subset G\cdot (i,j)$.
    Indeed, given such an $i_0$, for every $t\in G\cdot i=G\cdot i_0$ take $\sigma_t\in G$ with $\sigma_t(i_0)=t$.
    Then
    \[
        \{t\}\times[m]=\sigma_t(\{i_0\}\times[m])\subset G\cdot (i,j).
    \]

    Assume by contradiction that for every $t\in G\cdot i$ there exists $j_t\in [m]$ such that $(t,j_t)\in \Omega^I\setminus (G\cdot (i,j))$.
    Set
    \[
        v=(v_1,\ldots,v_r)\in \Omega^P,\qquad 
        v_t=\begin{cases}
            j_t,& t\in G\cdot i,\\
            1,& t\in \Omega^T\setminus (G\cdot i).
        \end{cases}
    \]
    Since $(G,\Omega^P)$ is transitive, there exists $\sigma\in G$ such that $\sigma(j,\ldots,j)=v$.
    Then
    \[
        \sigma(i,j)\in \{(t,j_t):t\in G\cdot i\}\subset \Omega^I\setminus (G\cdot (i,j)),
    \]
    a contradiction.
\end{proof}

\begin{rem}
\label{rem:primitive-transitive-does-not-imply-top-transitive}
    For $G\le R$, transitivity of $(G,\Omega^P)$ does not imply transitivity of $(G,\Omega^I)$ or $(G,\Omega^T)$.
    For example, the base group $S_m^r\le R$ acts transitively on $\Omega^P$, fixes every point of $\Omega^T$ and has $r$ orbits on $\Omega^I$.
\end{rem}

\begin{lemma}
\label{lem:induced-subgroups-from-top-orbits}
    Let $G\le R$ with $(G,\Omega^P)$ transitive.
    Let $G\cdot i_1,\ldots,G\cdot i_s$ be the distinct orbits of $(G,\Omega^T)$.
    Fix $1\le \ell\le s$ and set $r_\ell:=|G\cdot i_\ell|$.
    Then the subset $(G\cdot i_\ell)\times[m]\subset \Omega^I$ is $(G,\Omega^I)$-stable.
    Moreover, let $(G_\ell,(G\cdot i_\ell)\times[m])$ denote the permutation representation of $(G,(G\cdot i_\ell)\times[m])$.
    Then there exists $\Gamma_\ell\le S_m\wr S_{r_\ell}$, with $(\Gamma_\ell, \Omega^I(m,r_\ell))$ transitive, such that
    \[
        (G_\ell,(G\cdot i_\ell)\times[m])\cong (\Gamma_\ell,\Omega^I(m,r_\ell)).
    \]
\end{lemma}

\begin{proof}
    Fix $1\le \ell\le s$.
    By Lemma~\ref{lem:primitive-transitive-imprimitive-orbit}, $(G\cdot i_\ell)\times[m]$ is an orbit of $G$, hence $(G,\Omega^I)$-stable.
    Moreover, $(G_\ell,(G\cdot i_\ell)\times[m])$ is transitive and preserves the block system $\left\{\{i\}\times[m]:i\in G\cdot i_\ell\right\}$.
    By the wreath product embedding theorem \cite[Theorem~2.6]{meldrum1995wreath},
    \[
        (G_\ell,(G\cdot i_\ell)\times[m])\cong (\Gamma_\ell,\Omega^I(m,r_\ell)),
    \]
    for some $\Gamma_\ell\le S_m\wr S_{r_\ell}$, as claimed.
\end{proof}

\begin{lemma}
\label{lem:imprimitive-stabilizer-primitive-stabilizer}
    Let $G\le R$.
    For $(i,j)\in \Omega^I$ define
    \[
        \Lambda_{(i,j)}:=\{v\in \Omega^P:v_i=j\},\qquad \Lambda:=\{\Lambda_{(i,j)}:(i,j)\in \Omega^I\}.
    \]
    The action $(G,\mathcal{P}(\Omega^P))$ induces an action $(G,\Lambda)$, and the map $(i,j)\mapsto \Lambda_{(i,j)}$ defines an isomorphism $(G,\Omega^I)\cong (G,\Lambda)$.
\end{lemma}

\begin{proof}
    For $(i,j)\neq (k,\ell)$, the sets $\Lambda_{(i,j)}$ and $\Lambda_{(k,\ell)}$ are distinct.
    Indeed, if $i=k$ and $j\neq \ell$, then any $v\in \Lambda_{(i,j)}$ satisfies $v_k=v_i=j\neq \ell$, so $v\notin \Lambda_{(k,\ell)}$.
    If $i\neq k$, let $u\in [m]\setminus \{\ell\}$.
    Define $v\in \Omega^P$ by
    \[
        v_t:=\begin{cases}
            j,& t=i,\\
            u,& t=k,\\
            1,& \textnormal{otherwise}.
        \end{cases}
    \]
    Then $v\in \Lambda_{(i,j)}\setminus \Lambda_{(k,\ell)}$.
    Therefore the map $\iota:\Omega^I\to \Lambda$, $(i,j)\mapsto \Lambda_{(i,j)}$, is injective.

    Fix $\sigma=(\tau,\rho)\in G$ and $(i,j)\in \Omega^I$.
    Write $\sigma(i,j)=(k,\ell)$.
    If $v\in \Lambda_{(i,j)}$, then $v_i=j$, and therefore
    \[
        (\sigma v)_k=(\sigma v)_{\rho(i)}=\tau_{\rho(i)}\bigl(v_{\rho^{-1}(\rho(i))}\bigr)=\tau_{\rho(i)}(v_i)=\tau_{\rho(i)}(j)=\ell.
    \]
    Hence $\sigma v\in \Lambda_{(k,\ell)}=\Lambda_{\sigma(i,j)}$, so $\sigma(\Lambda_{(i,j)})\subset \Lambda_{\sigma(i,j)}$.
    By Lemma~\ref{lem:equivariant-family-isomorphism}, $\Lambda$ is $(G,\mathcal{P}(\Omega^P))$-stable and $(\mathrm{id},\iota)$ is an isomorphism.
\end{proof}

\subsection*{$k$-homogeneous actions}

\begin{lemma}
\label{lem:k-homogeneous-implies-transitive}
    Let $m\ge 2$, $1\le k\le m/2$, and $G\le S_m$.
    If $(G,\binom{[m]}{k})$ is transitive, then $(G,[m])$ is transitive.
\end{lemma}

\begin{proof}
    Assume $(G,[m])$ is not transitive.
    Choose $i\in [m]$ such that $|G\cdot i|\le m/2$.
    Choose sets $A,B\in \binom{[m]}{k}$ such that $i\in A$ and $B\subset [m]\setminus (G\cdot i)$.
    Since $(G,\binom{[m]}{k})$ is transitive, there exists $\sigma\in G$ such that $\sigma(A)=B$, hence $\sigma(i)\in B\subset [m]\setminus (G\cdot i)$, a contradiction.
\end{proof}

\begin{lemma}
\label{lem:point-stabilizer-equals-containing-subsets-stabilizer}
    Let $m\ge 2$, $1\le k\le m/2$, and $G\le S_m$.
    For $i\in [m]$ define
    \[
        \Lambda_i:=\{A\in \binom{[m]}{k}: i\in A\},\qquad \Lambda:=\{\Lambda_i:i\in [m]\}.
    \]
    The action $(G,\mathcal{P}\binom{[m]}{k})$ induces an action $(G,\Lambda)$, and the map $i\mapsto \Lambda_i$ defines an isomorphism $(G,[m])\cong (G,\Lambda)$.
\end{lemma}

\begin{proof}
    For $i\neq j$, the subsets $\Lambda_i$ and $\Lambda_j$ are distinct.
    Indeed, since $k<m$, there exists $A\in \binom{[m]}{k}$ such that $i\in A$ and $j\notin A$.
    Then $A\in \Lambda_i$ but $A\notin \Lambda_j$.
    Therefore the map $\iota:[m]\to \Lambda$, $i\mapsto \Lambda_i$, is injective.
    For $\sigma\in G$ and $i\in [m]$, clearly $\sigma(\Lambda_i)\subset \Lambda_{\sigma(i)}$.
    By Lemma~\ref{lem:equivariant-family-isomorphism}, $\Lambda$ is $(G,\mathcal{P}\binom{[m]}{k})$-stable and $(\mathrm{id},\iota)$ is an isomorphism.
\end{proof}

\subsection*{$k$-transitive actions}

\begin{lemma}
\label{lem:tuple-first-coordinate-stabilizer}
    Let $m\ge k\ge 1$ and $G\le S_m$.
    For $i\in [m]$ define
    \[
        \Lambda_i:=\{v\in [m]_k: v_1=i\},\qquad \Lambda:=\{\Lambda_i:i\in [m]\}.
    \]
    The action $(G,\mathcal{P}([m]_k))$ induces an action $(G,\Lambda)$, and the map $i\mapsto \Lambda_i$ defines an isomorphism $(G,[m])\cong (G,\Lambda)$.
\end{lemma}

\begin{proof}
    For $i\neq j$, the sets $\Lambda_i$ and $\Lambda_j$ are distinct.
    Indeed, every $v\in \Lambda_i$ satisfies $v_1=i$, so $v\notin \Lambda_j$.
    Therefore the map $\iota:[m]\to \Lambda$, $i\mapsto \Lambda_i$, is injective.
    For $\sigma\in G$ and $i\in [m]$, clearly $\sigma(\Lambda_i)\subset \Lambda_{\sigma(i)}$.
    By Lemma~\ref{lem:equivariant-family-isomorphism}, $\Lambda$ is $(G,\mathcal{P}([m]_k))$-stable and $(\mathrm{id},\iota)$ is an isomorphism.
\end{proof}

\subsection*{Regular actions}

\begin{lemma}
\label{lem:regular-action-evaluation-fibers}
    Let $(G,\Omega)\cong(G,G)$ be a regular permutation action, with $G$ acting on itself by left multiplication.
    Let $\mu=\mu(G)$ be the minimal faithful permutation degree, and identify $G$ as a subgroup of $S_\mu$.
    Let $G\cdot i_1,\ldots,G\cdot i_s$ be the distinct orbits of $(G,[\mu])$.
    For each $1\le \ell\le s$ and $i\in G\cdot i_\ell$, define
    \[
        \Lambda_i:=\{\sigma\in G: \sigma(i_\ell)=i\}.
    \]
    Set
    \[
        \Lambda^{(\ell)}:=\{\Lambda_i:i\in G\cdot i_\ell\},\qquad \Lambda:=\bigcup_{\ell=1}^s \Lambda^{(\ell)}.
    \]
    The action $(G,\mathcal{P}(\Omega))$ induces an action $(G,\Lambda)$, and the map $i\mapsto \Lambda_i$ defines an isomorphism $(G,[\mu])\cong (G,\Lambda)$.
\end{lemma}

\begin{proof}
    Fix $1\le \ell\le s$.
    For distinct $i,j\in G\cdot i_\ell$, the sets $\Lambda_i$ and $\Lambda_j$ are distinct.
    Indeed, if $\sigma\in \Lambda_i$, then $\sigma(i_\ell)=i\neq j$, so $\sigma\notin \Lambda_j$.
    Therefore the map $\iota_\ell:G\cdot i_\ell\to \Lambda^{(\ell)}$, $i\mapsto \Lambda_i$, is injective.
    For $\sigma\in G$, $i\in G\cdot i_\ell$, and $\tau\in \Lambda_i$, $(\sigma\tau)(i_\ell)=\sigma(i)$, so $\sigma\Lambda_i\subset \Lambda_{\sigma(i)}$.
    By Lemma~\ref{lem:equivariant-family-isomorphism}, $\Lambda^{(\ell)}$ is $(G,\mathcal{P}(\Omega))$-stable and $(\mathrm{id},\iota_\ell)$ is an isomorphism $(G,G\cdot i_\ell)\cong(G,\Lambda^{(\ell)})$.

    Let $1\neq \sigma\in G$.
    Choose $i\in [\mu]$ with $\sigma(i)\neq i$ and let $\ell$ be such that $i\in G\cdot i_\ell$.
    Then $\sigma(i)\in G\cdot i_\ell$, so $\sigma\Lambda_i=\Lambda_{\sigma(i)}\neq \Lambda_i$.
    It follows that $(G,\Lambda)$ is faithful.
    By minimality of $\mu$, $|\Lambda|\ge\mu$.
    On the other hand, $|\Lambda|\le\sum_{\ell=1}^s|G\cdot i_\ell|=\mu$, so $|\Lambda|=\mu$.
    Therefore the map $i\mapsto\Lambda_i$ is a bijection $[\mu]\to\Lambda$, and the per-orbit isomorphisms give an isomorphism $(G,[\mu])\cong(G,\Lambda)$.
\end{proof}

\section{Counting polynomials by group actions}
\label{sec:counting}

\subsection*{Definitions and notation}

For a separable polynomial $f\in\ZZ[x]$, let $\Omega_f:=\{\alpha\in\CC:f(\alpha)=0\}$ denote its roots, $K_f:=\QQ(\Omega_f)$ its splitting field, and $G_f:=\Gal(K_f/\QQ)$.
Since $f$ is separable, $G_f$ acts faithfully on $\Omega_f$, giving a permutation group $(G_f,\Omega_f)$.
Given a permutation group $(G,\Omega)$, set
\[
    B(H;G,\Omega):=\{f\in B_{|\Omega|}(H):(G_f,\Omega_f)\cong (G,\Omega)\}.
\]
In particular, every $f\in B(H;G,\Omega)$ is separable.
For $f\in B(H;G,\Omega)$, identify $(G_f,\Omega_f)=(G,\Omega)$, so that $\Omega$ is the set of roots of $f$.

If $K/\QQ$ is a Galois extension and $G\le \Gal(K/\QQ)$, write $K^G$ for the fixed field of $G$.
For fields $K_1,\ldots,K_s\subset \CC$, denote by $K_1\cdots K_s$ their compositum.

For an algebraic integer $b$, denote by $p_b\in \ZZ[x]$ its minimal polynomial over $\QQ$ and set $d_b:=\deg p_b=[\QQ(b):\QQ]$.
The absolute multiplicative Weil height of $b$ is given by
\[
    H_\omega(b):=\left(\prod_{\alpha\in\Omega_{p_b}}\max(1,|\alpha|)\right)^{1/d_b},
\]
see \cite[Proposition 1.6.6]{bombieri2006heights}.

\begin{lemma}
\label{lem:coefficients-of-stable-orbit-generate-subfield}
    Let $(G,\Omega)$ be a permutation group, and let $f\in B(H;G,\Omega)$. 
    Fix $A\subset \Omega$ and set
    \[
        f_A:=\prod_{\alpha\in A}(x-\alpha)=x^{|A|}+\sum_{i=0}^{|A|-1}b_i^{(A)} x^i.
    \]
    Then $K_f^{\Stab_G(A)}=\QQ\bigl(b_0^{(A)},\ldots,b_{|A|-1}^{(A)}\bigr)$.
\end{lemma}

\begin{proof}
    Let $\sigma\in G$. Then, since $f$ is separable,
    \begin{align*}
        \sigma\in\Stab_G(A)
        &\iff \sigma f_A=f_A \\
        &\iff \forall i:\sigma\bigl(b_i^{(A)}\bigr)=b_i^{(A)} \\
        &\iff \sigma\in \Gal\bigl(K_f/\QQ(b_0^{(A)},\ldots,b_{|A|-1}^{(A)})\bigr).
    \end{align*}
    The result then follows by Galois correspondence.
\end{proof}

The following theorem of Mahler \cite{mahler1962some} relates the Weil height of $b$ and the height of its minimal polynomial.
\begin{thm}[Mahler]
\label{thm:mahler}
    Fix $d\in \NN$. 
    There exist constants $c_1,c_2>0$ depending only on $d$ such that for every algebraic integer $b$ of degree $d_b\le d$ over $\QQ$, with minimal polynomial $p_b$,
    \[
        c_1 H(p_b) \le H_\omega(b)^{d_b}\le c_2 H(p_b).
    \]
\end{thm}

\subsection*{Splitting field factorization}
\begin{prop}
\label{prop:splitting-field-factorization-from-group-actions}
    Let $(G,\Omega)$ be a transitive permutation group.
    Let $\Lambda\subset\mathcal{P}(\Omega)$ be a $(G,\mathcal{P}(\Omega))$-stable family.
    Let $\Lambda^{(1)},\ldots,\Lambda^{(s)}$ be the distinct $G$-orbits in $(G,\Lambda)$.
    For each $\ell$, let $(G_\ell,\Lambda^{(\ell)})$ be the permutation representation of $(G,\Lambda^{(\ell)})$.
    Assume that $(G,\Lambda)$ is faithful.
    For each $\ell$, fix $A_\ell\in \Lambda^{(\ell)}$ and put 
    \[
        \nu_\ell:=\frac{|\Lambda^{(\ell)}|\cdot|A_\ell|}{|\Omega|}.
    \]

    Then there exists a constant $c>0$ such that for every $H>0$ and $f\in B(H;G,\Omega)$, there exist polynomials $q_1,\ldots,q_s$ with $K_f=K_{q_1}\cdots K_{q_s}$ and
    \[
        q_\ell\in B(cH^{\nu_\ell};G_\ell,\Lambda^{(\ell)}).
    \]
\end{prop}

\begin{proof}
    Fix $H>0$ and $f\in B(H;G,\Omega)$.
    For $A\in \Lambda$, define $E_A:=K_f^{\Stab_G(A)}$ and
    \[
        f_A:=\prod_{\alpha\in A}(x-\alpha)=x^{|A|}+\sum_{i=0}^{|A|-1} b_i^{(A)} x^i.
    \]
    By Lemma~\ref{lem:coefficients-of-stable-orbit-generate-subfield}, $E_A=\QQ\bigl(b_0^{(A)},\ldots,b_{|A|-1}^{(A)}\bigr)$.
 
    By definition $\sigma f_A=f_{\sigma(A)}$ for all $\sigma\in G$ and $A\in \Lambda$.
    By Galois correspondence,
    \[
        [E_{A_\ell}:\QQ]=|G/\Stab_G(A_\ell)|=|\Lambda^{(\ell)}|.
    \]
    Set
    \[
        F_\ell:=\prod_{\sigma\in G/\Stab_G(A_\ell)} \sigma f_{A_\ell} = \prod_{A\in \Lambda^{(\ell)}} f_A.
    \]
    Since $\sigma F_\ell=F_\ell$ for all $\sigma\in G$, it follows that $F_\ell\in \QQ[x]$.

    Since $(G,\Omega)$ is transitive, $f$ is irreducible.
    By construction, every root of $F_\ell$ is a root of $f$, and therefore $F_\ell=f^{n_\ell}$ for some $n_\ell\ge 0$.
    Comparing degrees,
    \[
        n_\ell = \frac{\deg F_\ell}{\deg f} = \frac{|\Lambda^{(\ell)}|\cdot|A_\ell|}{|\Omega|} = \nu_\ell,
    \]
    hence $F_\ell=f^{\nu_\ell}$.
    
    Fix $0\le i \le |A_\ell|-1$.
    By \cite[Lemma~4.4]{bary2024probabilistic} applied to the identity $F_\ell=\prod_{A\in \Lambda^{(\ell)}} f_A$,
    \[
        H_\omega\left(b_i^{(A_\ell)}\right) \ll H(f^{\nu_\ell})^{1/|\Lambda^{(\ell)}|} \ll H^{\nu_\ell/|\Lambda^{(\ell)}|}.
    \]
    By \cite[Lemma~4.5]{bary2024probabilistic}, there exists an algebraic integer $\beta_\ell$ with $E_{A_\ell}=\QQ(\beta_\ell)$ and
    \[
        H_\omega(\beta_\ell)\ll H^{\nu_\ell/|\Lambda^{(\ell)}|}.
    \]
    Let $q_\ell$ be the minimal polynomial of $\beta_\ell$.
    Since the implied constants in \cite[Lemmas~4.4 and~4.5]{bary2024probabilistic} and Theorem~\ref{thm:mahler} depend only on the degrees, there exists $c_\ell>0$, depending only on $(G,\Omega)$ and $\Lambda$, such that $H(q_\ell)\le c_\ell H^{\nu_\ell}$.
    Taking $c:=\max_{\ell} c_\ell$ gives, for all $\ell$,
    \[
        q_\ell\in B(cH^{\nu_\ell};G_\ell,\Lambda^{(\ell)}).
    \]

    Finally, for each $\ell$, the field $K_{q_\ell}$ is the compositum in $K_f$ of the fields $E_A$ with $A\in \Lambda^{(\ell)}$.
    By Galois correspondence,
    \[
        \Gal(K_f/K_{q_\ell})=\bigcap_{A\in \Lambda^{(\ell)}}\Stab_G(A)=\Fix_G(\Lambda^{(\ell)}).
    \]
    Since $(G,\Lambda)$ is faithful,
    \[
        \Gal\left(K_f/(K_{q_1}\cdots K_{q_s})\right)=\bigcap_{\ell=1}^s \Fix_G(\Lambda^{(\ell)})=\Fix_G(\Lambda)=1,
    \]
    therefore $K_f=K_{q_1}\cdots K_{q_s}$.
\end{proof}

\begin{prop}
\label{prop:counting-via-splitting-field-factorization}
    Fix $\delta>0$. 
    Let $(G,\Omega)$ be a transitive permutation group.
    Fix $s\in \NN$ and transitive permutation groups $(G_1,\Omega_1),\ldots,(G_s,\Omega_s)$.
    Let $\mathcal{H}_1,\ldots,\mathcal{H}_s:(0,\infty)\to (0,\infty)$ be functions.
    For $H>0$, set
    \[
        \mathcal{K}(H):=\{K_f: f\in B(H;G,\Omega)\}.
    \]
    Assume that for every $H>0$ and every $K\in \mathcal{K}(H)$ there exist polynomials $q_1,\ldots,q_s$ such that $K=K_{q_1}\cdots K_{q_s}$ and
    \[
        q_\ell\in B(\mathcal{H}_\ell(H);G_\ell,\Omega_\ell)
    \]
    for all $1\le \ell\le s$.
    Then
    \[
        |B(H;G,\Omega)|\ll H^{1+\delta}\prod_{\ell=1}^s \left|B(\mathcal{H}_\ell(H);G_\ell,\Omega_\ell)\right|.
    \]
\end{prop}

\begin{proof}
    Partition $B(H;G,\Omega)$ according to the splitting field,
    \[
        |B(H;G,\Omega)|=\sum_{K\in \mathcal{K}(H)}\left|\{f\in B(H;G,\Omega):K_f=K\}\right|.
    \]

    Write $n:=|\Omega|$. 
    Fix $K\in \mathcal{K}(H)$ and $f\in B(H;G,\Omega)$ with $K_f=K$.
    Let $\alpha$ be a root of $f$.
    Since $(G,\Omega)$ is transitive, $f$ is irreducible, so $\alpha$ is an algebraic integer with minimal polynomial $f$ and degree $d_\alpha=[\QQ(\alpha):\QQ]=n$.
    Put $L:=\QQ(\alpha)\subset K$.
    By Theorem~\ref{thm:mahler}, there exists $C_n>0$ depending only on $n$ such that
    \[
        H_\omega(\alpha)\le C_n H(f)^{1/n}\le C_n H^{1/n}.
    \]
    By \cite[Lemma 4.7]{bary2024probabilistic}, for every number field $L$ of degree $n$ with ring of integers $\OO_L$,
    \[
        \left|\left\{\alpha\in \OO_L:H_\omega(\alpha)\le C_n H^{1/n}\right\}\right|\ll H^{1+\delta},
    \]
    where the implied constant depends only on $n$ and $\delta$.
    Each such $\alpha$ determines $f$ as its minimal polynomial.
    Therefore, for each fixed subfield $L\subset K$ of degree $n$, there are $\ll H^{1+\delta}$ possibilities for $f$ with $\QQ(\alpha)=L$ and $H(f)\le H$.

    Since $\Gal(K/\QQ)\cong G$, Galois correspondence implies that $K$ has at most $C_G$ subfields of degree $n$, for some constant $C_G>0$ depending only on $G$.
    Therefore,
    \[
        \left|\{f\in B(H;G,\Omega):K_f=K\}\right|\ll H^{1+\delta},
    \]
    uniformly in $K\in \mathcal{K}(H)$. 
    Consequently,
    \[
        |B(H;G,\Omega)|\ll H^{1+\delta}|\mathcal{K}(H)|.
    \]

    Finally, by hypothesis, for every $K\in \mathcal{K}(H)$ there exist polynomials $q_1,\ldots,q_s$ such that $K=K_{q_1}\cdots K_{q_s}$ and $q_\ell\in B(\mathcal{H}_\ell(H);G_\ell,\Omega_\ell)$. 
    Therefore,
    \[
        |\mathcal{K}(H)|\le \prod_{\ell=1}^s \left|B(\mathcal{H}_\ell(H);G_\ell,\Omega_\ell)\right|,
    \]
    as claimed.
\end{proof}

\section{Proof of main results}
\label{sec:proofs}
All four proofs share the same structure.
Given a transitive action $(G,\Omega)$, we use the lemmas of Section~\ref{sec:group-actions} to construct a faithful $(G,\mathcal{P}(\Omega))$-stable family $\Lambda$, then apply Propositions~\ref{prop:splitting-field-factorization-from-group-actions} and~\ref{prop:counting-via-splitting-field-factorization} to reduce the count to simpler actions, and conclude using known bounds.

\subsection*{Primitive wreath product actions}
Fix $m,r\ge 2$, $1\le k\le m/2$, and $G\le S_m\wr S_r$.

\begin{lemma}
\label{lem:imprimitive-count}
    Fix $\delta>0$. Assume $(G,\Omega^I)$ is transitive. Then
    \[
        \left|B(H;G,\Omega^I)\right| \ll H^{m+r-1+\delta}. 
    \]
\end{lemma}
\begin{proof}
    Following \cite{bary2024probabilistic}, such a group $G$ is called transitive $(m,r)$-imprimitive.
    Let $\PP_{f\in B_{mr}(H)}$ denote the uniform probability on $B_{mr}(H)$.
    By \cite[Theorem 1.6]{bary2024probabilistic},
    \begin{align*}
        \frac{\left|B(H;G,\Omega^I)\right|}{\left|B_{mr}(H)\right|}
        &=\PP_{f\in B_{mr}(H)}\bigl((G_f,\Omega_f)\cong (G,\Omega^I)\bigr)\\
        &\le \PP_{f\in B_{mr}(H)}\bigl(G_f \textnormal{ is transitive }(m,r)\textnormal{-imprimitive}\bigr)\\
        &\ll H^{-mr+\mu+\delta},
    \end{align*}
    where $\mu=r-\frac{1}{r}+1$ if $m=2$ and $mr\geq 10$, and
    \[
        \mu=\max\left\{m+\frac{1}{2}+\frac{1}{r},r+(1-\frac{1}{r})(m-1),r+\frac{m}{2}\right\},
    \]
    otherwise.
    Since $\left|B_{mr}(H)\right|\asymp H^{mr}$ and $\mu\le m+r-1$ in all cases, the result follows.
\end{proof}

\begin{proof}[Proof of Theorem~\ref{thm:primitive-wreath-product}]
    It suffices to prove the theorem for $k=1$.
    Indeed, let $2\le k\le m/2$ and set $N=\binom{m}{k}$.
    Fix a bijection $\varphi:\binom{[m]}{k}\to [N]$.
    Conjugation by $\varphi$ gives an isomorphism $(S_m,\binom{[m]}{k})\cong (S_m,[N])$, hence an embedding $S_m\hookrightarrow S_N$.
    This extends to an embedding $S_m\wr S_r\hookrightarrow S_N\wr S_r$.
    Since the primitive action is defined coordinatewise from the action of $S_m$ on $\binom{[m]}{k}$, this embedding identifies
    \[
        (S_m\wr S_r,\Omega^P(m,r,k))\cong (S_m\wr S_r,\Omega^P(N,r,1))\hookrightarrow(S_N\wr S_r,\Omega^P(N,r,1)).
    \]
    In particular, $(G,\Omega^P(m,r,k))\cong (G,\Omega^P(N,r,1))$ is transitive, so the case $k=1$ with $m$ replaced by $N$ implies the general case.
    For the remainder of the proof assume $k=1$.

    Set $\Omega=\Omega^P=[m]^r$.
    By assumption $(G,\Omega)$ is transitive.
    For $(i,j)\in\Omega^I$, set
    \[
        \Lambda_{(i,j)}:=\{v\in\Omega:v_i=j\},\qquad \Lambda:=\{\Lambda_{(i,j)}:(i,j)\in\Omega^I\}.
    \]
    By Lemma~\ref{lem:imprimitive-stabilizer-primitive-stabilizer}, $(G,\Omega^I)\cong (G,\Lambda)$.
    By Lemma~\ref{lem:imprimitive-action-faithful}, $(G,\Lambda)$ is faithful.

    Let $G\cdot i_1,\ldots,G\cdot i_s$ be the distinct $G$-orbits in $\Omega^T$.
    Write $r_\ell:=|G\cdot i_\ell|$, so that $\sum_{\ell=1}^s r_\ell=r$.
    By Lemma~\ref{lem:primitive-transitive-imprimitive-orbit}, the $G$-orbits in $\Omega^I$ are $(G\cdot i_\ell)\times[m]$.
    Hence the $G$-orbits in $\Lambda$ are
    \[
        \Lambda^{(\ell)}:=\{\Lambda_{(i,j)}:(i,j)\in (G\cdot i_\ell)\times[m]\}.
    \]
    
    Since $|\Lambda^{(\ell)}|=mr_\ell$ and $|\Lambda_{(i_\ell,1)}|=m^{r-1}$,
    \[
        \nu_\ell:=\frac{|\Lambda^{(\ell)}|\cdot|\Lambda_{(i_\ell,1)}|}{|\Omega|}=\frac{(mr_\ell)m^{r-1}}{m^r}=r_\ell.
    \]
    By Proposition~\ref{prop:splitting-field-factorization-from-group-actions}, applied with $(G,\Omega)$ and $\Lambda$, there exists a constant $c>0$ such that for every $H>0$ and $f\in B(H;G,\Omega)$ there exist polynomials $q_1,\ldots,q_s$ with $K_f=K_{q_1}\cdots K_{q_s}$ and
    \[
        q_\ell\in B(cH^{r_\ell};G_\ell,\Lambda^{(\ell)}).
    \]

    By Lemma~\ref{lem:induced-subgroups-from-top-orbits}, for each $\ell$ there exists $\Gamma_\ell\le S_m\wr S_{r_\ell}$ such that $(\Gamma_\ell,\Omega^I(m,r_\ell))$ is transitive and
    \[
        (G_\ell,\Lambda^{(\ell)})\cong (\Gamma_\ell,\Omega^I(m,r_\ell)).
    \]

    Fix $\epsilon>0$ and put $\delta=\epsilon/(r+1)$.
    Apply Proposition~\ref{prop:counting-via-splitting-field-factorization} with $s$, $(G_\ell,\Omega_\ell)=(\Gamma_\ell,\Omega^I(m,r_\ell))$, and $\mathcal{H}_\ell(H)=cH^{r_\ell}$.
    It follows that
    \[
        |B(H;G,\Omega)|\ll H^{1+\delta}\prod_{\ell=1}^s\left|B(cH^{r_\ell};\Gamma_\ell,\Omega^I(m,r_\ell))\right|.
    \]
    By Lemma~\ref{lem:imprimitive-count}, for each $\ell$,
    \[
        \left|B(cH^{r_\ell};\Gamma_\ell,\Omega^I(m,r_\ell))\right|\ll H^{r_\ell(m+r_\ell-1+\delta)}.
    \]
    Since $r_\ell\le r$ and $\sum_{\ell=1}^s r_\ell=r$,
    \[
        \prod_{\ell=1}^s H^{r_\ell(m+r_\ell-1+\delta)}
        =H^{\sum_{\ell=1}^s r_\ell(m+r_\ell-1+\delta)}
        \le H^{\sum_{\ell=1}^s r_\ell(m+r-1+\delta)}
        =H^{r(m+r-1+\delta)}.
    \]
    Therefore,
    \[
        |B(H;G,\Omega)|\ll H^{1+\delta}H^{r(m+r-1+\delta)}=H^{r(m+r-1)+1+\epsilon},
    \]
    as claimed.
\end{proof}

\subsection*{$k$-homogeneous actions}

\begin{proof}[Proof of Theorem~\ref{thm:k-homogeneous}]
    Let $G\le S_m$ and assume that $(G,\binom{[m]}{k})$ is transitive. 
    Set $\Omega:=\binom{[m]}{k}$.
    By Lemma~\ref{lem:k-homogeneous-implies-transitive}, $(G,[m])$ is transitive.
    For $i\in[m]$, set
    \[
        \Lambda_i:=\{A\in\Omega: i\in A\},\qquad \Lambda:=\{\Lambda_i:i\in[m]\}.
    \]
    By Lemma~\ref{lem:point-stabilizer-equals-containing-subsets-stabilizer}, $(G,[m])\cong (G,\Lambda)$.

    By Proposition~\ref{prop:splitting-field-factorization-from-group-actions}, applied with $(G,\Omega)$, $\Lambda$, $\Lambda^{(1)}=\Lambda$, and
    \[
        \nu_1=\frac{|\Lambda^{(1)}|\cdot|\Lambda_1|}{|\Omega|}
        =\frac{m\binom{m-1}{k-1}}{\binom{m}{k}}=k,
    \]
    there exists a constant $c>0$ such that for every $H>0$ and $f\in B(H;G,\Omega)$ there exists $q\in B(cH^k;G,[m])$ with $K_f=K_q$.

    Fix $\epsilon>0$. Apply Proposition~\ref{prop:counting-via-splitting-field-factorization} with $s=1$, $(G_1,\Omega_1)=(G,[m])$, and $\mathcal{H}_1(H)=cH^k$.
    It follows that
    \[
        |B(H;G,\Omega)|\ll H^{1+\epsilon} |B(cH^k;G,[m])|\le H^{1+\epsilon} |B_m(cH^k)|\ll H^{mk+1+\epsilon},
    \]
    as claimed.
\end{proof}

\subsection*{$k$-transitive actions}

\begin{proof}[Proof of Theorem~\ref{thm:k-transitive}]
    Set $\Omega:=[m]_k$.
    Let $G\le S_m$ be $k$-transitive.
    In particular $(G,[m])$ is transitive.
    For $i\in[m]$, set
    \[
        \Lambda_i:=\{v\in\Omega: v_1=i\},\qquad \Lambda:=\{\Lambda_i:i\in[m]\}.
    \]
    By Lemma~\ref{lem:tuple-first-coordinate-stabilizer}, $(G,[m])\cong (G,\Lambda)$.

    By Proposition~\ref{prop:splitting-field-factorization-from-group-actions}, applied with $(G,\Omega)$, $\Lambda$, $\Lambda^{(1)}=\Lambda$, and
    \[
        \nu_1=\frac{|\Lambda^{(1)}|\cdot|\Lambda_1|}{|\Omega|}
        =\frac{m\cdot (m-1)_{k-1}}{(m)_k}=1,
    \]
    there exists a constant $c>0$ such that for every $H>0$ and $f\in B(H;G,\Omega)$ there exists $q\in B(cH;G,[m])$ with $K_f=K_q$.

    Fix $\epsilon>0$.
    Apply Proposition~\ref{prop:counting-via-splitting-field-factorization} with $s=1$, $(G_1,\Omega_1)=(G,[m])$, and $\mathcal{H}_1(H)=cH$.
    It follows that
    \[
        |B(H;G,\Omega)|\ll H^{1+\epsilon}|B(cH;G,[m])|\le H^{1+\epsilon}|B_m(cH)|\ll H^{m+1+\epsilon},
    \]
    as claimed.
\end{proof}

\subsection*{Regular actions}

\begin{proof}[Proof of Theorem~\ref{thm:regular}]
    Let $(G,\Omega)\cong(G,G)$ be a regular permutation group, with $G$ acting on itself by left multiplication.
    Let $\mu=\mu(G)$ be the minimal faithful permutation degree, and identify $G$ as a subgroup of $S_\mu$.
    Let $G\cdot i_1,\ldots,G\cdot i_s$ be the distinct $G$-orbits in $[\mu]$.
    For each $1\le\ell\le s$ and $i\in G\cdot i_\ell$, set
    \[
        \Lambda_i:=\{\sigma\in\Omega:\sigma(i_\ell)=i\},\qquad \Lambda^{(\ell)}:=\{\Lambda_i:i\in G\cdot i_\ell\}.
    \]
    Put $\Lambda:=\bigcup_{\ell=1}^s\Lambda^{(\ell)}$.
    By Lemma~\ref{lem:regular-action-evaluation-fibers}, $(G,[\mu])\cong(G,\Lambda)$.
    Since $\left|\Lambda^{(\ell)}\right|=|G \cdot i_\ell|$, and $\left|\Lambda_{i_\ell}\right|=\left|\Stab_G(i_\ell)\right|$, by the orbit-stabilizer theorem,
    \[
        \nu_\ell:=\frac{|\Lambda^{(\ell)}|\cdot|\Lambda_{i_\ell}|}{|\Omega|}=\frac{|G\cdot i_\ell|\cdot|\Stab_G(i_\ell)|}{|G|}=1.
    \]

    By Proposition~\ref{prop:splitting-field-factorization-from-group-actions}, applied with $s$, $(G,\Omega)$ and $\Lambda$, there exists a constant $c>0$ such that for every $H>0$ and $f\in B(H;G,\Omega)$ there exist polynomials $q_1,\ldots,q_s$ with $K_f=K_{q_1}\cdots K_{q_s}$ and
    \[
        q_\ell\in B(cH;G_\ell,\Lambda^{(\ell)}).
    \]

    Fix $\epsilon>0$.
    Apply Proposition~\ref{prop:counting-via-splitting-field-factorization} with $s$, $(G_\ell,\Omega_\ell)=(G_\ell,\Lambda^{(\ell)})$, and $\mathcal{H}_\ell(H)=cH$.
    It follows that
    \[
        |B(H;G,\Omega)|\ll H^{1+\epsilon}\prod_{\ell=1}^s \left|B(cH;G_\ell,\Lambda^{(\ell)})\right|.
    \]
    Since $|\Lambda|=\sum_{\ell=1}^s |\Lambda^{(\ell)}|=\mu$,
    \[
        \prod_{\ell=1}^s \left|B(cH;G_\ell,\Lambda^{(\ell)})\right|
        \le \prod_{\ell=1}^s |B_{|\Lambda^{(\ell)}|}(cH)|
        \ll H^{\sum_{\ell=1}^s |\Lambda^{(\ell)}|}= H^\mu.
    \]
    Therefore,
    \[
        |B(H;G,\Omega)|\ll H^{\mu+1+\epsilon},
    \]
    as claimed.
\end{proof}

\bibliographystyle{plain}
\bibliography{references}

@article{bary2024probabilistic,
  title={Probabilistic Galois Theory: The square discriminant case},
  author={Bary-Soroker, Lior and Ben-Porath, Or and Matei, Vlad},
  journal={Bulletin of the London Mathematical Society},
  volume={56},
  number={6},
  pages={2162--2177},
  year={2024},
  publisher={Wiley Online Library}
}

@article{bhargava2025galois,
  title={Galois groups of random integer polynomials and van der Waerden's Conjecture},
  author={Bhargava, Manjul},
  journal={Annals of Mathematics},
  volume={201},
  number={2},
  pages={339--377},
  year={2025},
  publisher={Department of Mathematics, Princeton University Princeton, New Jersey, USA}
}

@article{van1936seltenheit,
  title={Die Seltenheit der reduziblen Gleichungen und der Gleichungen mit Affekt},
  author={van der Waerden, Bartel L},
  journal={Monatshefte f{\"u}r Mathematik und Physik},
  volume={43},
  number={1},
  pages={133--147},
  year={1936},
  publisher={Springer}
}

@book{bombieri2006heights,
  title={Heights in Diophantine geometry},
  author={Bombieri, Enrico and Gubler, Walter},
  number={4},
  year={2006},
  publisher={Cambridge university press}
}

@article{mahler1962some,
  title={On some inequalities for polynomials in several variables},
  author={Mahler, Kurt},
  journal={J. London Math. Soc},
  volume={37},
  number={1},
  pages={341--344},
  year={1962}
}

@book{meldrum1995wreath,
  title={Wreath products of groups and semigroups},
  author={Meldrum, John DP},
  volume={74},
  year={1995},
  publisher={CRC Press}
}

@article{babai1993faithful,
  title={On faithful permutation representations of small degree},
  author={Babai, L{\'a}szl{\'o} and Goodman, Albert J and Pyber, L{\'a}szl{\'o}},
  journal={Communications in Algebra},
  volume={21},
  number={5},
  pages={1587--1602},
  year={1993},
  publisher={Taylor \& Francis}
}

@article{lemke2020upper,
  title={Upper bounds on polynomials with small Galois group},
  author={Lemke Oliver, Robert J and Thorne, Frank},
  journal={Mathematika},
  volume={66},
  number={4},
  pages={1054--1059},
  year={2020},
  publisher={Wiley Online Library}
}

@article{chela1963reducible,
  title={Reducible polynomials},
  author={Chela, R},
  journal={Journal of the London Mathematical Society},
  volume={1},
  number={1},
  pages={183--188},
  year={1963},
  publisher={Wiley Online Library}
}

@inproceedings{knobloch1955hilbertschen,
  title={Zum hilbertschen irreduzibilit{"a}tssatz},
  author={Knobloch, Hans-Wilhelm},
  booktitle={Abhandlungen aus dem Mathematischen Seminar der Universit{"a}t Hamburg},
  volume={19},
  number={3},
  pages={176--190},
  year={1955},
  organization={Springer}
}

@inproceedings{gallagher1973large,
  title={The large sieve and probabilistic {G}alois theory},
  author={Gallagher, Patrick X},
  booktitle={Proc. Sympos. Pure Math},
  volume={24},
  pages={91--101},
  year={1973}
}

@article{zywina2010hilbert,
  title={Hilbert's irreducibility theorem and the larger sieve},
  author={Zywina, David},
  journal={arXiv preprint arXiv:1011.6465},
  year={2010}
}

@article{dietmann2013probabilistic,
  title={Probabilistic {G}alois theory},
  author={Dietmann, Rainer},
  journal={Bulletin of the London Mathematical Society},
  volume={45},
  number={3},
  pages={453--462},
  year={2013},
  publisher={Wiley Online Library}
}

@article{chow2023towards,
  title={Towards van der Waerden's conjecture},
  author={Chow, Sam and Dietmann, Rainer},
  journal={Transactions of the American Mathematical Society},
  volume={376},
  number={04},
  pages={2739--2785},
  year={2023}
}

@article{anderson2023quantitative,
  title={Quantitative {H}ilbert irreducibility and almost prime values of polynomial discriminants},
  author={Anderson, Theresa C and Gafni, Ayla and Lemke Oliver, Robert J and Lowry-Duda, David and Shakan, George and Zhang, Ruixiang},
  journal={International Mathematics Research Notices},
  volume={2023},
  number={3},
  pages={2188--2214},
  year={2023},
  publisher={Oxford University Press}
}

@article{livingstone1965transitivity,
  title={Transitivity of finite permutation groups on unordered sets},
  author={Livingstone, Donald and Wagner, Ascher},
  journal={Mathematische Zeitschrift},
  volume={90},
  number={5},
  pages={393--403},
  year={1965},
  publisher={Springer}
}

\end{document}